\documentclass[12pt]{amsart}
\usepackage{geometry} 
\usepackage{amsmath,amsthm,amsfonts,amssymb}
\usepackage{graphicx}
\usepackage[utf8]{inputenc}

\newtheorem{lemma}{Lemma}[section]
\newtheorem{proposition}[lemma]{Proposition}
\newtheorem{theorem}[lemma]{Theorem}
\newtheorem{conjecture}[lemma]{Conjecture}
\newtheorem{corollary}[lemma]{Corollary}

\theoremstyle{definition}
\newtheorem{definition}[lemma]{Definition}
\newtheorem{remark}[lemma]{Remark}

\newcommand{\Z}{\mathbb{Z}}

\newcommand{\Q}{\mathbb{Q}}
\newcommand{\C}{\mathbb{C}}

\title{Algebraic generators of the skein algebra of a surface}

\author{Ramanujan Santharoubane}

\date{\today}

\begin{document}
\address{Department of Mathematics, University of Virginia, Charlottesville, VA 22904-4137, USA}
\email{ramanujan.santharoubane@gmail.com}
 
\maketitle
\begin{abstract} Let $\Sigma$ be a surface with negative Euler characteristic, genus at least one and at most one boundary component. We prove that the Kauffman bracket skein algebra of $\Sigma$ over the field of rational functions can be algebraically generated by a finite number of simple closed curves that are naturally associated to certain generators of the mapping class group of $\Sigma$. The action of the mapping class group on the skein algebra gives canonical relations between these generators. From this, we conjecture a presentation for a skein algebra of $\Sigma$.

\end{abstract}
\section{Introduction}
\subsection{Main results}
This paper is focused on finding algebraic generators of the Kauffman bracket skein algebra of a surface. Throughout this paper, we will refer to the Kauffman bracket skein algebra simply as the skein algebra. Let $\Sigma$ be a compact oriented connected surface of genus at least one and with at most one boundary component. Moreover we will suppose that $\Sigma$ has negative Euler characteristic. We denote by $\mathcal{S}(\Sigma,\Q(A))$ the skein module of $\Sigma \times [0,1]$ with coefficients in the field of rational function $\Q(A)$ and by $S(\Sigma)$ the skein module over $\Z[A^{\pm 1}]$. The module $\mathcal{S}(\Sigma,\Q(A))$ is equipped with a natural product given by stacking banded links. For $\gamma$ is a simple closed curve on $\Sigma$, we write $\gamma$ for the element $\gamma \times \left[\frac{1}{2},\frac{2}{3}\right]$ in $\mathcal{S}(\Sigma,\Q(A))$ and we denote by $t_{\gamma}$ the Dehn twist along $\gamma$. 

\begin{theorem} \label{theorem1}
Let $\{ \gamma_j \}_{j \in I}$ be a finite set of non separating simple closed curves such that 
\begin{enumerate}
\item For any $i,j \in I$, the curves $\gamma_i$ and $\gamma_j$ intersect at most once.
\item The set $\{ t_{\gamma_j} \}_{j \in I}$ generates the mapping class group of $\Sigma$.
\end{enumerate}
Then the algebra generated by $\{\gamma_j\}_{j \in I}$ is $\mathcal{S}(\Sigma,\Q(A))$. Moreover if $\xi$ is a non zero complex number such that $\xi^4 \neq 1$ then the set $\{\gamma_j\}_{j \in I}$ generates $\mathcal{S}(\Sigma)\underset{A=\xi}{\otimes} \C$ as an algebra.
\end{theorem}
We recall that the mapping class group of $\Sigma$ is $\pi_0(\mathrm{Homeo}^{\small{+}}(\Sigma,\partial \Sigma))$.
We will now give an interpretation of some relations that should hold for the generators in the previous theorem.  Let us fix $\{ \gamma_j \}_{j \in I}$ a set of simple closed curves on $\Sigma$ satisfying the hypothesis of Theorem \ref{theorem1}. Let $\mathbb{Q}(A) \langle I \rangle $ be the free non commutative $\mathbb{Q}(A)$-algebra generated by $\{ X_j \}_{j \in I}$. The theorem says that there exists a surjective algebra homomorphism \begin{equation} \label{psi} \Psi : \mathbb{Q}(A) \langle I \rangle  \to \mathcal{S}(\Sigma,\Q(A))\end{equation} defined by: $$\Psi(X_j)=\frac{\gamma_j}{A^2-A^{-2}} \quad \forall j \in I$$
Now for $j \in I$ and $\epsilon \in \{-1,1\}$, let $T_j ^{\epsilon} : \mathbb{Q}(A) \langle I \rangle  \to \mathbb{Q}(A) \langle I \rangle $ be the homomorphism of $\Q(A)$-algebra defined by:
\begin{align}  
T_j^{\epsilon} (X_k) & = X_k  &\text{when} \, \,  \iota (\gamma_j , \gamma_k)= 0 \label{eq1} \\
 T_j^{\epsilon} (X_k) & =  \epsilon (A^{\epsilon} X_j X_k-A^{-\epsilon} X_k X_j)  &\text{when} \, \,  \iota (\gamma_j , \gamma_k)= 1 \label{eq2} \end{align}
Here $\iota$ is the geometric intersection of simple closed curves. We can easily check that $\Psi$ exchanges the actions of the $T_j$ and the $t_j$ in the sense that 
$$ \Psi(T_j^{\epsilon} X) = t_{\gamma_j}^{\epsilon} (\Psi(X)) \quad \quad \text{for all} \, \, X \in \mathbb{Q}(A) \langle I \rangle,  j \in I,\, \text{and} \, \, \epsilon\in\{-1,1\}$$
Let $\Gamma(\Sigma)$ be the mapping class group of $\Sigma$ and let $\bar{\Gamma}(\Sigma)$ be the group $\Gamma(\Sigma)$  modulo its center. Suppose that $I = \{1, \ldots , N\}$ and let us denote each $t_{\gamma_j}$ simply by $t_j$. 
Note that the map $t_j^{\epsilon} \mapsto T_j^{\epsilon}$ does not extend to an action of $\bar{\Gamma}(\Sigma)$ on $\mathbb{Q}(A) \langle I \rangle$. Indeed $\mathbb{Q}(A) \langle I \rangle$ is a free non commutative algebra and the relations in $\bar{\Gamma}(\Sigma)$ satisfied by the Dehn twists $\{t_j^{\pm1}\}_{j \in J}$ are clearly not satisfied by the operators $\{T_j^{\pm 1}\}_{j \in J}$. Hence the relations between the $\{t_j^{\pm1}\}_{j \in J}$ give relations between the generators $\{\gamma_j\}_{j \in J}$ in $\mathcal{S}(\Sigma,\Q(A))$. Suppose that $\bar{\Gamma}(\Sigma)$ has the following presentation with respect to the generators $\{t_j \}_{j \in I}$: 
$$\bar{\Gamma}(\Sigma) = \bigl\langle \, t_1, \ldots , t_N \, | \, R_1(t_1, \ldots, t_N) = \cdots = R_K(t_1, \ldots, t_N) = 1 \bigr\rangle $$ where $K$ is an integer and the $R_k(t_1,\ldots,t_N)$ are some words in $\{t_j^{\pm 1} \}_{j\in I}$. Let $\mathcal{R}$ be the bi-ideal of $\mathbb{Q}(A) \langle I \rangle $ generated by the elements

\begin{align}
& R_k (T_1, \ldots, T_N) X_{i} -X_i,  \quad \text{for} \, \, 1 \le i \le N \, \text{and} \, 1 \le k \le K  \notag \\
& T_j T_{j}^{-1} X_i-X_i,  \quad \text{for}\, \, 1 \le i,j  \le N   \notag  \\
&X_i X_j - X_j X_i,   \quad   \text{for} \, \,  \iota (\gamma_i , \gamma_j)= 0 \notag
\end{align}
We define \begin{equation} \label{eq3} \mathcal{A}(\Gamma(\Sigma)) =\frac{\mathbb{Q}(A) \langle I \rangle }{ \mathcal{R}} \end{equation} which is a quotient of $\mathbb{Q}(A) \langle I \rangle$ on which the actions of the $T_j^{\pm 1}$ extend to a canonical action of $\bar{\Gamma}(\Sigma)$.
A direct consequence of Theorem \ref{theorem1} is the following:
\begin{corollary}
The canonical map $$\mathcal{A}(\Gamma(\Sigma)) \to \mathcal{S}(\Sigma,\Q(A))$$ is surjective.

\end{corollary}

\begin{conjecture} \label{c}
There exists a presentation of $\bar{\Gamma}(\Sigma)$ for which $\mathcal{A}(\Gamma(\Sigma))$ is isomorphic to $\mathcal{S}(\Sigma,\Q(A))$ as a non commutative $\Q(A)$-algebra.

\end{conjecture}

\subsection{Notes and references}

Bullock in \cite{B} was the first to find algebraic generators of the skein algebra of a surface. His generators are over $\Z[A^{\pm 1}]$ and not over $\Q(A)$. The number of his generators is exponential in the genus of the surface whereas here we have a linear number (by chosing the right generators of $\Gamma(\Sigma)$).

Finite generation was also prove by Abdiel and Frohman (see \cite[Thm 3.7]{AF}). In \cite{FB}, Frohman and Kania-Bartoszynska studied the skein algebra when $A$ is evaluated at a root of unity. They prove that it is generated over its center by a pair of subalgebras from pants decomposition. Their generators have some similarities with the one in the current paper.

Presentations of skein algebras of surfaces are only known in genus zero and one. In \cite{BP}, Bullock and Przytycki found such a presentation for the one holed torus, the four holed sphere and twice holed torus. They related some of these algebras to non standard deformations of lie algebras.

When $A$ is specialized to $-1$, it was shown by Bullock (see \cite{B97}) and Przytycki-Sikora (see \cite{PS}) that the skein algebra of a surface is isomorphic to the ring of algebraic functions of the $\mathrm{SL}(2,\C)$ character variety of the surface. Moreover, for $A = \sqrt{-1}$, March\'e (see \cite{Mar}) gave an homological interpretation of the skein algebra of the surface. Note that the map $\Psi$ defined in Equation (\ref{psi}) is not defined if $A$ is specialized to a $4$-th primitive root of unity. It is possible to see that if we specialize $A$ at a $4$-th root of unity in the algebra $\mathcal{A}(\Gamma(\Sigma))$ we find something different from the algebras studied by  Bullock, March\'e and Przytycki-Sikora.

Humphries generators (see \cite{H77}) and Lickorish generators (see \cite{L64}) are examples of generators of the mapping class groups satisfying the hypothesis of Theorem \ref{theorem1}. Moreover presentations for both of these generating sets are know, we refer to book \cite{Farb} of Farb and Margalit for more details.

We consider $\Gamma(\Sigma)$ quotiented by its center because the center of the mapping class group acts trivially on the skein algebra of the surface.

\section{Acknowledgements}
I want to thank C. Frohman, J. March\'e and G. Masbaum for helpful conversations.

\section{Quick review of the skein algebra}
 For any compact oriented manifold $M$ (maybe with boundary), we denote by $\mathcal{S}(M)$ the Kauffman bracket skein module with coefficients in $\Z[A^{\pm 1}]$. We recall that it is the free $\Z[A^{\pm1}]$-module generated by the set of isotopy classes of banded links in the interior of $M$ quotiented by the following so-called skein relations.  First
 \begin{equation} \label{KI} L_{\times} = A  L_{\infty} + A^{-1}  L_0 \end{equation} where $L_{\times}$, $L_0$, $ L_{\infty}$ are any three banded links in $M$ which are the same outside a small $3$-ball but differ inside as in Figure \ref{Ktriple}. In this case, the triple $L_{\times}$, $L_0$, $ L_{\infty}$ is called a Kauffman triple. The second relation satisfied in $\mathcal{S}(M)$ is $$L \cup D=-(A^2+A^{-2}) \, L$$ where $L$ is any link in $M$ and $D$ is a trivial banded knot. We define $\mathcal{S}(M, \Q(A))$ to be the $\Q(A)$-vector space $\mathcal{S}(M)\underset{}{\otimes} \Q(A)$ where the tensor product is made over $\Z[A^{\pm 1}]$.

\begin{figure}[htbp]
\includegraphics[width=3.2cm]{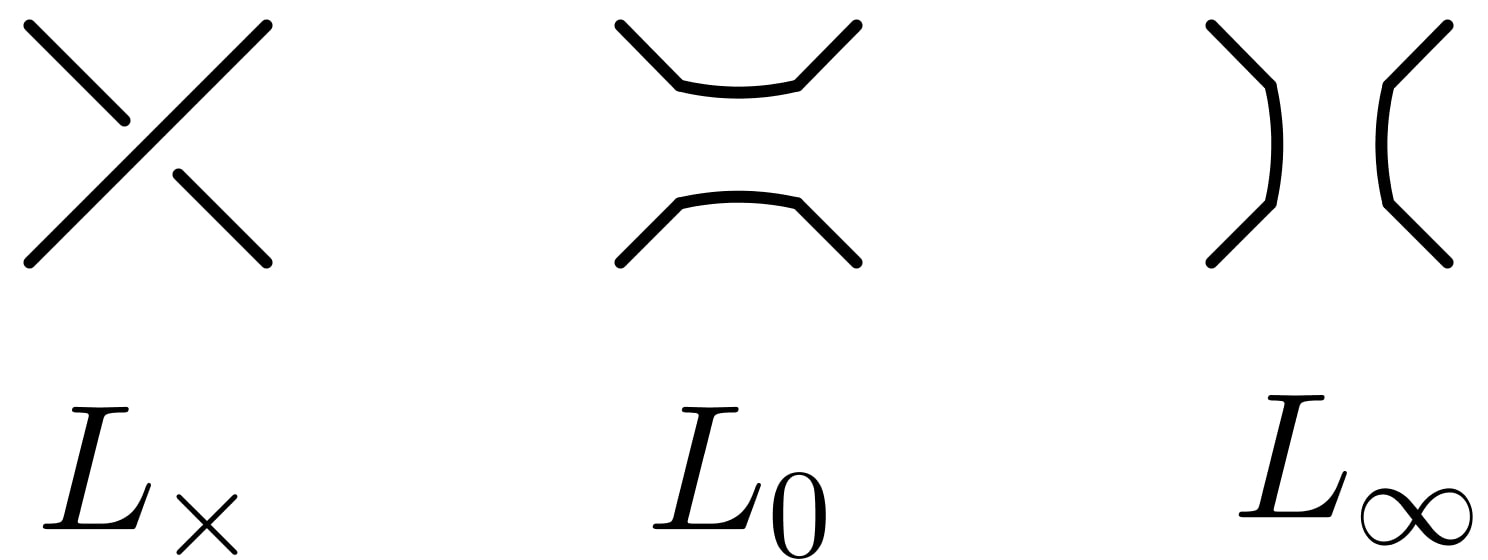}
\caption{Kauffman triple}\label{Ktriple}
\end{figure}

Let $\Sigma$ be a compact oriented connected surface, we denote by $\mathcal{S}(\Sigma,\Q(A))$ the space $\mathcal{S}(\Sigma \times [0,1],\Q(A))$. Stacking banded links on top of each other gives $\mathcal{S}(\Sigma,\Q(A))$ a structure of $\Q(A)$-algebra.

A multicurve is a disjoint union of non null-homotopic simple close curve inside $\Sigma$. For $\gamma \subset \Sigma$ a multicurve we write $\gamma$ for the banded link $\gamma \times \left[\frac{1}{2},\frac{2}{3}\right]$ in $\mathcal{S}(\Sigma,\Q(A))$ and we will still call this banded link a multicurve. A well known theorem is the following:

\begin{theorem}
The set of isotopy classes of multicurves on $\Sigma$ is a basis of the $\Q(A)$-vector space $\mathcal{S}(\Sigma,\Q(A))$.
\end{theorem}
In particular this theorem clearly implies that simple closed curves generate $\mathcal{S}(\Sigma,\Q(A))$ as an algebra. Recall that $\Gamma(\Sigma) = \pi_0(\mathrm{Homeo}^{\small{+}}(\Sigma,\partial \Sigma))$ acts canonically on  $\mathcal{S}(\Sigma,\Q(A))$. Recall also that if $\gamma \subset \Sigma$ is a simple closed curve, we denote by $t_{\gamma}$ the Dehn twist along $\gamma$. The following lemma is standard and can be obtained by applying the skein relation (\ref{KI}).
\begin{lemma} \label{sk}
Let $\alpha$ and $\beta$ be two simple close curves intersecting once. We have in  $\mathcal{S}(\Sigma,\Q(A))$:
$$ t_{\alpha}^{\epsilon} (\beta) =  \dfrac{A^{\epsilon} \alpha \beta-A^{-\epsilon} \beta \alpha}{\epsilon(A^2-A^{-2})}$$
\end{lemma}

\section{Proof of Theorem \ref{theorem1}}
Let $\Gamma(\Sigma)$ be the mapping class group of $\Sigma$. Let $\{ \gamma_j \}_{j \in I}$ be a set of simple closed curves satisfying the hypothesis of the Theorem \ref{theorem1} and let $\mathcal{B}$ be the sub-algebra of $\mathcal{S}(\Sigma,\Q(A))$ generated by $\{\gamma_j\}_{j \in I}$.
\begin{lemma} \label{MCG}
 $\mathcal{B}$ is stable by the action of $\Gamma(\Sigma)$.
 \end{lemma}
 \begin{proof}
 Since $\{ t_{\gamma_j} \}_{j \in I}$ generate $\Gamma(\Sigma)$, it enough to prove that for any $j,k \in I$ we have $t_{\gamma_j}^{\pm 1}(\gamma_k) \in \mathcal{B}$. If $\gamma_j$ does not intersect $\gamma_k$ then $t_{\gamma_j}^{\pm 1}(\gamma_k) = \gamma_k \in \mathcal{B}$. Now if $\gamma_j$ intersects $\gamma_k$ once then, by Lemma \ref{sk}, we have $$t_{\gamma_j}^{\pm 1}(\gamma_k)= \frac{A^{\pm 1} \gamma_j \gamma_k-A^{\mp 1} \gamma_k \gamma_j}{\pm(A^2-A^{-2})} \in \mathcal{B}$$
\end{proof}

\begin{lemma}\label{nonsep}
If  $\mathcal{C}_0$ denotes the set of non separating simple closed curves then $\mathcal{C}_0 \subset \mathcal{B}$. 
\end{lemma}

\begin{proof}
Let $\gamma \in \mathcal{C}_0$, there exists $\phi \in \Gamma(\Sigma)$ such that $\phi(\gamma_1) = \gamma$. Since $\gamma_1$ belongs to $\mathcal{B}$ which is stable by the action of $\Gamma(\Sigma)$ (see previous lemma), we have $\gamma \in \mathcal{B}$.

\end{proof}

\begin{lemma}\label{sep}
If  $\mathcal{C}_1$ denotes the set of  separating simple closed curves then $\mathcal{C}_1 \subset \mathcal{B}$. 

\end{lemma}
\begin{proof}
Suppose that the genus of $\Sigma$ is $g \ge 1$. Let $\delta_1,\ldots,\delta_g$ be the following curves 
$$\includegraphics[scale = 0.45]{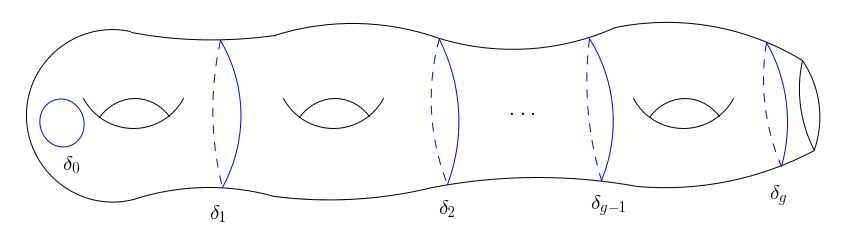}$$
where $\delta_g$ is trivial when $\Sigma$ does not have boundary. Let $j \in \{1,\ldots,g\}$ and let $z_j,z_j'$ be the two following non separating curves in the torus with two boundary components defined $\delta_j,\delta_{j-1}$ : 
$$\includegraphics[scale = 0.45]{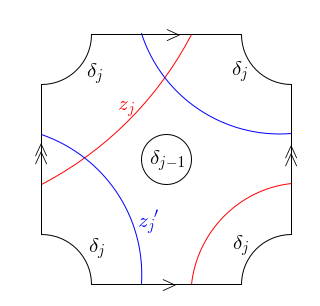}$$
By applying the skein relations we have :
$$ z_j' z_j = A^2 x_j x_j'+A^{-2} y_j y_j'+ \delta_j+ \delta_{j-1}$$ where $\delta_0= -A^2-A^{-2}$ and $x_j,x_j',y_j,y_j'$ are non separating curves. By Lemma \ref{nonsep}, $z_j', z_j, x_j,x_j',y_j,y_j' \in \mathcal{B}$ so by an induction on $j$ we can prove that for all $ 1 \le j \le g$ we have $\delta_j \in \mathcal{B}$.

Now if $\gamma$ is a separating curve, there exists $\phi \in \Gamma(\Sigma)$ and $j_0$ such that $\gamma = \phi(\delta_{j_0})$. Since $\mathcal{B}$ is stable by the action of $\Gamma(\Sigma)$, we have $\gamma \in\mathcal{B}$.

\end{proof}

\begin{proof}[Proof of Theorem \ref{theorem1}] 
$\mathcal{S}(\Sigma,\Q(A))$ is algebraically generated by simple closed curves. Hence combining Lemma \ref{nonsep} and Lemma \ref{sep}, we can conclude. Moreover if $\xi$ is a non zero complex number such that $\xi^2-\xi^{-2} \neq 0$, we can still conclude by Lemmas \ref{MCG}, \ref{nonsep} and \ref{sep} that $\mathcal{B} = S(\Sigma) \underset{A=\xi}{\otimes} \C$.

\end{proof}

\section{Interpretation of the relations in the skein algebra}
Let $\{ \gamma_j \}_{j \in I}$ be a set of simple closed curves on $\Sigma$ satisfying the hypothesis of Theorem \ref{theorem1}. Recall that $\mathbb{Q}(A) \langle I \rangle $ is the free non commutative $\mathbb{Q}(A)$-algebra generated by $\{ X_j \}_{j \in I}$. 
\begin{definition}
For $X,Y \in \mathbb{Q}(A) \langle I \rangle $ we define $[X,Y]_A := AXY - A^{-1} YX$.
\end{definition}
Recall that the maps $\{T_j^{\epsilon} \}_{j \in I}$ are defined by (\ref{eq1}) and (\ref{eq2}) in the introduction. Recall also that given a presentation of $\bar{\Gamma}(\Sigma)$ with respected to the generating set $\{t_{\gamma_j} \}_{j \in I}$, we defined $\mathcal{A}(\Gamma(\Sigma))$ via Equation (\ref{eq3}) (see the introduction). By definition, any relation satisfied by the $\{t_{\gamma_j} \}_{j \in I}$ (which may not appear in the given presentation) gives some relation in $\mathcal{A}(\Gamma(\Sigma))$.
Let us focus on the following relations: 
\begin{align}
&T_j^{-1} T_i^{-1} T_j^{-1} T_i T_j T_i X_{a} -X_a=0,  \quad \text{for} \, \,  \iota (\gamma_i , \gamma_j)= 1 \, \text{and} \, a \in I \label{braid} \\
& T_j T_{j}^{-1} X_i-X_i=0, \quad \quad \quad  \quad  \quad \, \, \,  \, \, \, \, \, \text{for}\, \, i,j \in I   \label{inv}  \\
&X_i X_j - X_j X_i=0,  \quad  \quad \quad \quad  \quad  \, \, \,  \, \,  \, \, \,   \,  \text{for} \, \,  \iota (\gamma_i , \gamma_j)= 0 \label{com}
\end{align}
Note that these relations hold but are not a complete set of relations in $\mathcal{A}(\Gamma(\Sigma))$. The first relation comes from the braid relations in the mapping class group. 
\begin{proposition} \label{braiding}
In $\mathcal{A}(\Gamma(\Sigma))$, the relation (\ref{braid}), (\ref{inv}) and (\ref{com}) are equivalent to:
\begin{align}
[[X_j,X_i]_A,X_j]_A = X_i ,  & \quad \text{for} \, \,  \iota (\gamma_i , \gamma_j)= 1  \label{coxeter} \\
X_i X_j - X_j X_i=0,  & \quad \text{for} \, \,  \iota (\gamma_i , \gamma_j)= 0 \label{com1}
\end{align}
\end{proposition}
\begin{proof}
Let $i,j \in I$, note that if $\iota(\gamma_i,\gamma_j) =0$, the relation (\ref{inv}) in empty and if $ \iota(\gamma_i,\gamma_j) =1$, this relation gives  $[[X_j,X_i]_A,X_j]_A = X_i$.

Let $i,j \in I$ such that $\iota(\gamma_i,\gamma_j) =1$. Because of (\ref{inv}), the relation (\ref{braid}) can be rewritten
$$T_i T_j T_i X_k = T_j T_i T_j X_k$$  for all $k \in I$.
It is easy to check that this relation is implied by (\ref{coxeter}) and (\ref{com1}).
\end{proof}
\begin{remark}
We did not include $T_j^{-1} T_j X_ i = X_i$ in the relations defining $\mathcal{A}(\Gamma(\Sigma))$ because they give $[X_j,[X_i,X_j]_A]_A = X_i$ for $\iota(\gamma_i,\gamma_j) =1$ which is the same as $[[X_j,X_i]_A,X_j]_A = X_i$.
\end{remark}

\subsection{The case of the one-holed torus}
Let $\Sigma$ be a surface of genus one surface with one boundary component. Its mapping class group is the braid group $B_3$ whose presentation is $\langle t_1,t_2 \, | \, t_1 t_2 t_1 = t_2 t_1 t_2 \rangle$. Here is $t_1$ is Dehn twist along the canonical meridian of $\Sigma$ and $t_2$ is the Dehn twist around the longitude of $\Sigma$. Note that these two curves satisfy the hypothesis of Theorem \ref{theorem1}. The center of this group is the group generated by $(t_1 t_2 t_1)^2$ and $\bar{\Gamma}(\Sigma)$ is $\mathrm{PSL}_2(\Z)$ with following presentation $$\bar{\Gamma}(\Sigma) = \langle t_1,t_2 \,| \,t_1 t_2 t_1 = t_2 t_1 t_2, \, (t_1 t_2 t_1)^2=1\rangle$$  In this case $\mathcal{A}(\Gamma(\Sigma))$ is a non commutative algebra generated by $X_1$ and $X_2$. Because of Proposition \ref{braiding}, the only relations between $X_1$ and $X_2$ are:
$$ [X_1,[X_2,X_1]_A]_A = X_2, \, [X_2,[X_1,X_2]_A]_A = X_1,\, (T_1 T_2 T_1)^2X_1 = X_1, \, (T_1 T_2 T_1)^2 X_2 = X_2  $$
It is easy to check that the two last relations are implied by the two first one. Therefore $$\mathcal{A}(\Gamma(\Sigma)) = \langle X_1, X_2 | [X_1,[X_2,X_1]_A]_A = X_2, \, [X_2,[X_1,X_2]_A]_A = X_1 \rangle$$ From \cite[Thm 2.1]{B} the skein module of the one-holed torus is isomorphic to $\mathcal{A}(\Gamma(\Sigma))$. Therefore Conjecture \ref{c} holds for the one holed torus.

\end{document}